%% file: main.tex
\documentclass[dvipsnames, a4paper, 11pt,abstract=true]{scrartcl}
\usepackage[margin=1in,bottom=1in]{geometry}
\usepackage{libertine}

\usepackage[utf8]{inputenc}
\usepackage[T1]{fontenc}

\usepackage{amsthm}
\usepackage{amsmath}

\usepackage{amssymb}
\usepackage{stmaryrd}
\usepackage{thm-restate}
\usepackage[colorlinks = true,linkcolor = blue,urlcolor = blue, citecolor = blue]{hyperref}
\usepackage{url}
\usepackage{todonotes}
\usepackage{soul}
\usepackage{relsize}
\usepackage{enumerate}
\usepackage[all,defaultlines=3]{nowidow}
\usepackage[mathlines]{lineno}
\usepackage{multicol}
\usepackage{xspace}
\usepackage{lineno}

\usepackage{textpos}

\usepackage{mathtools}
\usepackage{xcolor}
\usepackage{tikz}
\usetikzlibrary{calc}
\usetikzlibrary{intersections,through}
\usetikzlibrary{shapes,decorations}
\usetikzlibrary{decorations.pathmorphing}
\usetikzlibrary{hobby}
\usetikzlibrary{decorations.markings,intersections}

\pgfdeclarelayer{background}
\pgfdeclarelayer{foreground}
\pgfsetlayers{background,main,foreground}

\pgfkeys{%
	/tikz/on layer/.code={
		\pgfonlayer{#1}\begingroup
		\aftergroup\endpgfonlayer
		\aftergroup\endgroup
	},
	/tikz/node on layer/.code={
		\pgfonlayer{#1}\begingroup
		\expandafter\def\expandafter\tikz@node@finish\expandafter{\expandafter\endgroup\expandafter\endpgfonlayer\tikz@node@finish}%
	}
}

\usepackage{titling}
\usepackage{dsfont}
\usepackage{xspace}
\usepackage{stmaryrd}
\usepackage{mathrsfs}
\usepackage{hyperref}
\usepackage{cleveref}
\usepackage{comment}
\usepackage[inline]{enumitem}
\setlist{nosep}

\usepackage[subrefformat=parens]{subcaption} 
\usepackage{nicefrac}
\usepackage{marginnote}
\usepackage{macros}
\usepackage{tikzmacros}
\usepackage{textpos}

 \hypersetup{
 	colorlinks=true,
 	linkcolor=myViolet,
 	citecolor=myBlue,
 	urlcolor=myLightblue,
 	bookmarksopen=true,
 	bookmarksnumbered,
 	bookmarksopenlevel=2,
 	bookmarksdepth=3
 }

\newif\ifcomment
\commentfalse
\commenttrue

\newcommand{\EP}{Erdős-Pósa\xspace}

\title{\EP property of \tripods in directed graphs\thanks{%
  This work was initiated during STRUG: Stuctural Graph Theory Bootcamp, funded by the “Excellence initiative – research university (2020-2026)” of the University of Warsaw.
  Meike Hatzel's research was supported by the Federal Ministry of Education and Research (BMBF), by a fellowship within the IFI programme of the German Academic Exchange Service (DAAD) and by the Institute for Basic Science (IBS-R029-C1).
  Karolina Okrasa's research was supported by the Polish National Science Centre grant no. 2021/41/N/ST6/01507.
  Micha\l{} Pilipczuk's research was supported by the project BOBR that has received funding from the European Research Council (ERC) under the European Union’s Horizon 2020 research and innovation programme with grant agreement No. 948057.
  }}
\DeclareRobustCommand{\authorthing}{
Marcin Bria\'nski\thanks{Department of Theoretical Computer Science, Faculty of Mathematics and Computer Science, Jagiellonian University, Kraków, Poland. \url{marcin.brianski@doctoral.uj.edu.pl}}
\and
Meike Hatzel\thanks{Discrete Mathematics Group, Institute for Basic Science (IBS), Daejeon, South Korea.
\url{research@meikehatzel.com}}
\and
Karolina Okrasa\thanks{University of Oxford, United Kingdom \& Warsaw University of Technology, Poland. \url{karolina.okrasa@cs.ox.ac.uk}}
\and
Michał Pilipczuk\thanks{Institute of Informatics, University of Warsaw, Poland. \url{michal.pilipczuk@mimuw.edu.pl}}}
\author{\authorthing}
\date{}

\begin{document}

\maketitle

\begin{abstract}
	Let $D$ be a directed graphs with distinguished sets of sources $S\subseteq V(D)$ and sinks $T\subseteq V(D)$.
	A \emph{\tripod} in $D$ is a subgraph consisting of the union of two $S$-$T$-paths that have distinct start-vertices and the same end-vertex, and are disjoint apart from sharing a suffix.

	We prove that \tripods in directed graphs exhibit the \EP property.
	More precisely, there is a function $f\colon \N\to \N$ such that for every digraph $D$ with sources $S$ and sinks $T$, if $D$ does not contain $k$ vertex-disjoint \tripods, then there is a set of at most $f(k)$ vertices that meets all the \tripods in $D$.
\end{abstract}

\begin{textblock}{20}(-2.1,5.4)
   \includegraphics[width=80px]{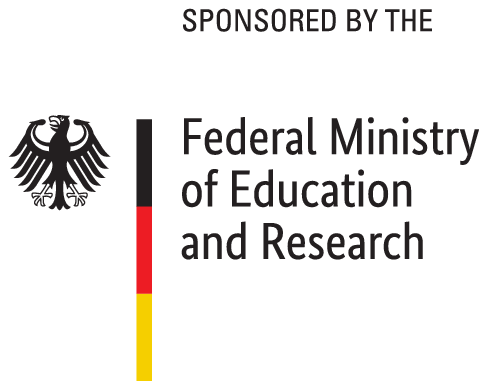}
\end{textblock}
\begin{textblock}{20}(-2.1,6.5)
 \includegraphics[width=60px]{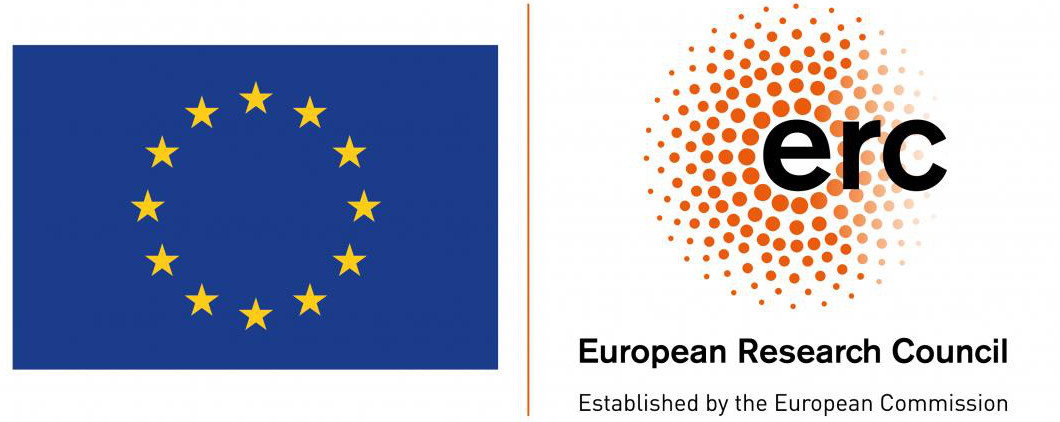}
\end{textblock}

\thispagestyle{empty}

\newpage

\setcounter{page}{1}

\input{intro}

\input{preliminaries}

\input{crossed}

\input{mainProof}

\input{conjecture}

\bibliographystyle{alpha}
\bibliography{bibliography}
\end{document}

%% file: intro.tex
\section{Introduction}\label{sec:intro}

As proved by Erdős and Pósa~\cite{EPtheorem1965} already in 1965, every undirected graph contains $k$ vertex-disjoint cycles or a set of $\Oh(k\log k)$ vertices that meets all the cycles. In modern terminology, this means that cycles in undirected graphs have the \emph{\EP property}: there is a functional relation between their hitting number and their packing number. Since the work of Erdős and Pósa, the search for families of objects in graphs exhibiting the \EP property has gathered significant interest; see the dynamic survey by Raymond~\cite{EPwebpageRaymond} for an extensive selection of both positive and negative results.

Arguably, questions about \EP properties become much more difficult in directed graphs (\emph{digraphs}, for short), and far fewer results are known. For instance, directed cycles in digraphs do have the \EP property, but this was established only in 1996 by Reed, Robertson, Seymour, and Thomas~\cite{ReedRST96} after being open for more than 20 years. More generally, Amiri, Kawarabayashi, Kreutzer, and Wollan~\cite{AmiriKKW16} proved that topological minor models of a fixed strongly connected digraph $H$ exhibit the \EP property as long as $H$ itself is a topological minor of a cylindrical wall. To this end, they combined the Directed Grid Theorem of Kawarabayashi and Kreutzer~\cite{KawarabayashiK15} with the classic width-based approach proposed by Robertson and Seymour~\cite{GM5} in their proof that in undirected graphs, minor models of any fixed planar graph $H$ have the \EP property.

The assumption about the strong connectivity of $H$ is used vitally in the proof of Amiri et al..
Intuitively, only under this assumption, a directed tree decomposition provides separation properties that are useful in the context of the \EP property of topological minor models of $H$. For instance, if the pattern graph $H$ were acyclic, then \EP-type question about models of $H$ would be meaningful and non-trivial already in directed acyclic graphs (DAGs), which have directed treewidth $0$. In fact, the authors are not aware of any non-trivial \EP-type results for acyclic patterns in general digraphs, except for the classic Menger's Theorem~\cite{menger1927,directedmenger}, which establishes the \EP property for $S$-$T$-paths in digraphs with prescribed sets of sources $S$ and sinks $T$.

In this work, we deliver one such result.
We study objects we call \emph{\tripods}, which can be regarded as generalisations of $S$-$T$-paths considered in Menger's Theorem. To define them, it is convenient to introduce some definitions.

A \emph{\migration digraph} is a digraph $D$ together with prescribed sets of \emph{sources} $S(D)\subseteq V(D)$ and of \emph{sinks} $T(D)\subseteq V(D)$. If $D$ is clear from the context, we often simply write $S,T$ instead of $S(D),T(D)$.
A \emph{\tripod}, see~\cref{fig:tripod}, in a \migration digraph $D$ is a subgraph of $D$ consisting of:
\begin{itemize}
 \item two distinct sources $s_1,s_2\in S$;
 \item a sink $t\in T$;
 \item a vertex $c\in V$, called the \emph{centre}; and
 \item an $s_1$-$c$-path, an $s_2$-$c$-path, and a $c$-$t$-path, all vertex-disjoint except for sharing $c$.
\end{itemize}
Note that any of the paths in the last point might have length $0$. In particular, an $S$-$T$-path that has an internal vertex belonging to $S$ is also a \tripod.

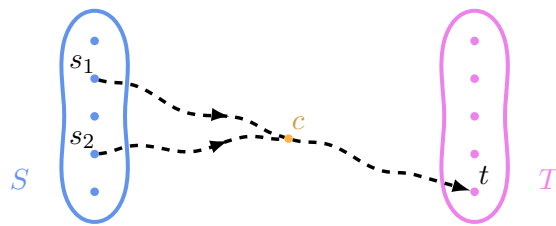
\begin{figure}[!ht]
	\centering
	\begin{tikzpicture}
		\node (centre) at (0,0) {};
		\def\dist{0.5}
		\node[vertex,CornflowerBlue] (s-3) at ($(centre)+(180:2.5)$) {};
		\node[vertex,CornflowerBlue] (s-2) at ($(s-3)+(90:\dist)$) {};
		\node (s-1-label) at ($(s-2)+(130:0.25)$) {$s_1$};
		\node[vertex,CornflowerBlue] (s-1) at ($(s-2)+(90:\dist)$) {};
		\node[vertex,CornflowerBlue] (s-4) at ($(s-3)+(270:\dist)$) {};
		\node (s-2-label) at ($(s-4)+(130:0.25)$) {$s_2$};
		\node[vertex,CornflowerBlue] (s-5) at ($(s-4)+(270:\dist)$) {};
		
		\node[vertex,PastelOrange] (t-centre) at ($(centre)+(280:0.3)$) {};
		\node[PastelOrange!85!black] (centre-label) at ($(t-centre)+(60:0.25)$) {$c$};
		
		\node[vertex,LavenderMagenta] (t-3) at ($(centre)+(0:2.5)$) {};
		\node[vertex,LavenderMagenta] (t-2) at ($(t-3)+(90:\dist)$) {};
		\node[vertex,LavenderMagenta] (t-1) at ($(t-2)+(90:\dist)$) {};
		\node[vertex,LavenderMagenta] (t-4) at ($(t-3)+(270:\dist)$) {};
		\node[vertex,LavenderMagenta] (t-5) at ($(t-4)+(270:\dist)$) {};
		\node (t-label) at ($(t-5)+(60:0.25)$) {$t$};
		
		\draw[pathdecorated] (s-2) to (t-centre);
		\draw[pathdecorated] (s-4) to (t-centre);
		\draw[path] (t-centre) to (t-5);
		
		\begin{pgfonlayer}{background}
			\def\setdist{0.4}
			\draw[line width=1.5pt,CornflowerBlue] ($(s-1)+(180:\setdist)$) to[closed,curve through ={($(s-1)+(90:\setdist)$) ($(s-1)+(0:\setdist)$) ($(s-1)!0.5!(s-5)+(0:\setdist)$) ($(s-5)+(0:\setdist)$) ($(s-5)+(270:\setdist)$) ($(s-5)+(180:\setdist)$)}] ($(s-5)!0.5!(s-1)+(180:\setdist)$);
			
			\node[CornflowerBlue] (S-label) at ($(s-5)+(170:1)$) {$S$};
			
			\draw[line width=1.5pt,LavenderMagenta] ($(t-1)+(180:\setdist)$) to[closed,curve through ={($(t-1)+(90:\setdist)$) ($(t-1)+(0:\setdist)$) ($(t-1)!0.5!(t-5)+(0:\setdist)$) ($(t-5)+(0:\setdist)$) ($(t-5)+(270:\setdist)$) ($(t-5)+(180:\setdist)$)}] ($(t-5)!0.5!(t-1)+(180:\setdist)$);
			
			\node[LavenderMagenta] (T-label) at ($(t-5)+(10:1)$) {$T$};
		\end{pgfonlayer}
	\end{tikzpicture}
	\caption{A \tripod with its \textcolor{PastelOrange!85!black}{centre $c$}, the two sources $s_1$ and $s_2$ and its sink $t$.}
	\label{fig:tripod}
\end{figure}


With these definitions, our main result reads as follows.

\begin{theorem}\label{thm:main}
 There is a function $\bound{thm:main}{f}{}\colon \N\to \N$ such that for every \migration digraph $D$ and $k\in \N$~either
 \begin{itemize}
  \item there is a family of $k$ vertex-disjoint \tripods in $D$, or
  \item there is a set of at most $\bound{thm:main}{f}{k}$ vertices that intersects every \tripod in $D$.
 \end{itemize}
\end{theorem}

Our strategy to prove \cref{thm:main} can be summarised as follows. In the following description, we fix a \migration digraph $D$ and $k\in \N$, and we use $\Omega_k(1)$ and $\Oh_k(1)$ as short-hands for a lower bound, respectively an upper bound, that is phrased as a function of $k$. For undefined terms, see \cref{sec:preliminaries}.
\begin{itemize}
 \item We use a recent result of Bożyk, Defrain, Okrasa, and Pilipczuk~\cite{bozyk2022digraphs}, called the \emph{Onion Harvesting Lemma}, to prove the following statement (\cref{lem:web_to_packing}): if in $D$ there are linkages $\Kk$ and $\Ll$ with $|\Kk|=|\Ll|=\Omega_k(1)$ such that the paths of $\Kk$ all start in $S$, the paths of $\Ll$ all end in~$T$, and every path of $\Kk$ intersects every path of $\Ll$, then there are $k$ vertex-disjoint \tripods in~$D$.
 \item We use the statement above to prove the following (\cref{thm:crossing_linkages}): if there are two $S$-$T$-linkages $\Pp$ and $\Qq$ with $|\Pp|=|\Qq|=\Omega_k(1)$ such that $\Pp$ and $\Qq$ have disjoint sets of start-vertices but the same set of end-vertices, then there are $k$ vertex-disjoint \tripods in $D$. We give two proofs: one using Ramsey arguments, and one exploiting an intermediate statement (\cref{lem:tripod_or_web}) that in order to find $k$ vertex-disjoint \tripods, it is enough to exhibit an $S$-$T$-linkage $\Pp$ of size $\Omega_k(1)$ and a path $Q$ that intersects every path of $\Pp$.
 \item Next, we prove (\cref{lem:cut_partition}) that unless there are $k$ vertex-disjoint \tripods in $D$, we have the following: for every partition $(S_1,S_2)$ of $S$, there is a partition $(T_1,T_2)$ of $T$ and a set of $\Oh_k(1)$ vertices that intersects all $S_1$-$T_2$-paths and all $S_2$-$T_1$-paths. This follows from combining the statement from the previous point with the Matroid Intersection Theorem of Edmonds~\cite{edmonds1970}.
 \item We finalise the proof of \cref{thm:main} by employing a variant of the classic separation-based framework for proving \EP-type results (see e.g.~\cite{bruce1999mangoes}), where the statement from the previous point serves the role of a ``small separation lemma''.
\end{itemize}

We remark that the appearance of the result of Bożyk et al.~\cite{bozyk2022digraphs} in this work is not a coincidence.
As the reader can see in the proof of \cref{lem:web_to_packing}, \tripods are just thinly disguised \emph{onions}: structures studied by Bożyk et al.~in the context of directed immersions.
In fact, \cref{thm:main} was devised as a technical statement within a larger project to prove a suitable grid theorem for directed immersions, which was initiated by Bożyk et al..
Thus, this work can be regarded as a follow-up of theirs, and we hope that \cref{thm:main} may be useful in a successful completion of said project.

The result we present in this paper considers vertex-disjoint \tripods, however one can directly infer a similar result on edge-disjoint \tripods as well.
\begin{corollary}\label{thm:main_edge}
	Let $k\in \N$ and $D$ be a \migration digraph $D$ in which every source $s\in S(D)$ has no incoming edges and one outgoing edge and every sink $t\in T(D)$ has one incoming edge and no outgoing edges. Then~either
	\begin{itemize}
		\item there is a family of $k$ edge-disjoint \tripods in $D$, or
		\item there is a set of at most $\bound{thm:main}{f}{k}$ edges that intersects every \tripod in $D$.
	\end{itemize}
\end{corollary}
To see that \cref{thm:main_edge} follows directly from \cref{thm:main}, we take the given graph $D$ construct its \emph{linegraph} $D'$: the digraph on the vertex set $\E{D}$ and with edges of the form $(e,f)$ where the head of $e$ is equal to the tail of $f$. We treat $D'$ as a \migration digraph by setting the edges outgoing from the sources of $D$ to be sources, and the edges incoming to the sinks of $D$ to be sinks. It is then straightforward to see that tripods in $D$ correspond to tripods in $D'$, where the correspondence translates edge-disjointness to vertex-disjointness. Hence it suffices to apply \cref{thm:main} in $D'$.



%% file: preliminaries.tex
\section{Preliminaries}
\label{sec:preliminaries}

For $k \in \N$ we write $[k]$ for the set of integers $\Set{1,\dots,k}$. A \emph{partition} of a set $A$ is a pair $(A_1,A_2)$ of subsets of $A$ such that $A_1\cap A_2=\emptyset$ and $A_1\cup A_2=A$. Note that we allow $A_1$ or $A_2$ to be~empty. In an undirected graph, an \emph{independent set} and a \emph{clique} are sets of pairwise non-adjacent, respectively pairwise adjacent, vertices.

A \emph{digraph} $D$ consists of a set of vertices $\V{D}$ and a set of edges $\E{D} \subseteq \V{D} \times \V{D}$.
For an edge $e = (x,y) \in \E{D}$ we call $\Head{e} \coloneqq y$ the \emph{head} of $e$ and $\Tail{e} \coloneqq x$ the \emph{tail} of $e$.
Thus, the digraphs in this paper are \emph{simple} --- do not contain multiple edges with same head and tail --- and we also assume them not to contain loops.

A \emph{subgraph} $D'$ of a digraph $D$ is a digraph with $\V{D'} \subseteq \V{D}$ and $\E{D'} \subseteq \E{D}.$

A (directed) path $P$ of length $\ell$ in $D$ is a sequence of distinct vertices $(v_0, v_1, \dots, v_{\ell})$ such that $\Brace{v_i, v_{i+1}} \in \E{D}$ for all $i \in [\ell]$.
We call $\Start{P} \coloneqq v_0$ the \emph{start-vertex} of $P$ and $\End{P} \coloneqq v_{\ell}$ the \emph{end-vertex} of $P$.
We say $P$ is a $\Start{P}$-$\End{P}$-path.
For sets $A,B \subseteq \V{D}$ we say $P$ is an $A$-$B$-path if $\Start{P} \in A$ and $\End{P} \in B$.
We sometimes identify a path $P$ with the subgraph consisting of its vertices and the connecting edges.

A \emph{linkage} in a digraph $D$ is a family of vertex-disjoint paths, and it is an \emph{$A$-$B$-linkage} if it consists of $A$-$B$-paths.
For a linkage $\mathcal{L}$, we write \[\Start{\mathcal{L}} \coloneqq \{\Start{P}\colon P\in \mathcal{L}\}\qquad\textrm{and}\qquad \End{\mathcal{L}} \coloneqq \{\End{P}\colon P\in \mathcal{L}\}.\]
We also use this notation for collection of paths that are not necessarily disjoint. Two linkages $\Pp$ and~$\Qq$ \emph{pairwise intersect} if every path of $\Pp$ intersects every path of $\Qq$.

\Migration digraphs and
\tripods were already defined in \cref{sec:intro}.
For a migration digraph $D$, a \emph{packing of \tripods} in $D$ is a family of vertex disjoint \tripods in $D$. By $\packing(D)$ we denote the largest size of a packing of \tripods in $D$.
Also, we say that a subset of vertices $F$ of a digraph $D$ \emph{meets} or \emph{intersects} a subgraph $D'$ of $D$ if $D'$ contains a vertex that belongs to $F$.
This terminology is applied to paths and to \tripods in particular.

Finally, we make use of the well-known result by Menger~\cite{menger1927,directedmenger} that the size of the minimum set of vertices separating two sets is equal to the maximum size of a linkage between them.

\begin{theorem}[Menger's Theorem~\cite{menger1927,directedmenger}]
	\label{thm:menger}
	Let $D$ be a digraph, $A$ and $B$ be subsets of vertices of $D$, and $k \in \N.$
	Then one of the following holds:
	\begin{itemize}
		\item there is a vertex subset of size at most $k-1$ that meets every $A$-$B$-path, or
		\item there is an $A$-$B$-linkage of order $k$.
	\end{itemize}
\end{theorem}

%% file: crossed.tex
\section{Packing \tripods}
\label{sec:crossed}
\label{sec:packing}

In this~\namecref{sec:packing} we prove the following statement.

\begin{theorem}
	\label{thm:crossing_linkages}
	There is a function $\bound{thm:crossing_linkages}{g}{}\colon \N \to \N$ with the following property.
	Let $D$ be a \migration digraph, $k$ be a positive integer,
  and $\Pp$ and $\Qq$ be two $S$-$T$-linkages of size at least $\bound{thm:crossing_linkages}{g}{k}$ such that $\End{\Pp} = \End{\Qq}$ and $\Start{\Pp} \cap \Start{\Qq} = \emptyset$.
	Then $D$ contains a packing of $k$ \tripods.
\end{theorem}

We give two ways to prove this result.
The first proof, presented in \cref{sec:ramsey}, uses Ramsey arguments and has the advantage of being concise and simple. The second proof, presented in \cref{sec:web_or_tripod}, contains an intermediate statement (\cref{thm:path_crossing_linkage}) of independent interest.
Both proofs make use of a result of Bożyk et al.~for directed immersions~\cite{bozyk2022digraphs}, which we translate into a statement about \tripods here.

\subsection{Onions}
\label{sec:onions}

To state the result from the work of Bożyk et al.~\cite{bozyk2022digraphs} we need a few definitions.
An \emph{onion}\footnote{As the original result speaks about immersions, we slightly adjust the definitions in order to be able to talk about subgraphs, by discussing paths instead of images of edges under an immersion embedding.} is a digraph $Z$ which consists of two distinct vertices $x$ and $y$, and three edge-disjoint paths, two of them from $x$ to $y$ and one of them from $y$ to $x$.
We call $x$ the \emph{root} and $y$ the \emph{\tip} of $Z$.
An \emph{onion-star} of order $k$ is a digraph consisting of $2k$ pairwise edge-disjoint onions such that $k$ of them share the same root~$x$ which is also the \tip of the remaining $k$ onions, and all other roots and \tips are distinct. The vertex $x$ is then called the \emph{centre} of the onion-star.

\input{figOnions}

\begin{lemma}[Bożyk, Defrain, Okrasa, Pilipczuk \cite{bozyk2022digraphs}]
	\label{lem:onion_harvesting}
	There exists a function $\bound{lem:onion_harvesting}{g}{} \colon \N \to \N$ such that the following holds.
	Let $D$ be a digraph, $x$ be a vertex of $D$, $t$ be a positive integer,
  and let $\Pp$ and~$\Qq$ be two families of paths in $D$, each of size at least $\bound{lem:onion_harvesting}{g}{t}$, satisfying the following:
	\begin{itemize}
	 \item $\Start{P} = \End{Q} = x$ for all $P \in \Pp$ and $Q\in \Qq$,
	 \item $E(P) \cap E(Q) \neq \emptyset$ for all $P \in \Pp$ and $Q \in \Qq$, and
	 \item the paths of $\Pp$ are pairwise edge-disjoint and the paths of $\Qq$ are pairwise edge-disjoint.
	\end{itemize}
 Then $D$ contains an onion-star of order $t$ with centre $x$.
\end{lemma}

We make use of this result to identify a situation in which we can expose a large packing of \tripods.

\begin{corollary}
	\label{lem:web_to_packing}
	Let $D$ be a \migration digraph, $k$ be a positive integer,
  $\Kk$ be a linkage with $\Start{\Kk} \subseteq S$, and $\Ll$ be a linkage with $\End{\Ll} \subseteq T$. Suppose $\Abs{\Kk} = \Abs{\Ll} \geq \bound{lem:onion_harvesting}{g}{k}$ and $\Kk$ and $\Ll$ pairwise intersect. Then $D$ contains a packing of $k$ \tripods.
\end{corollary}
\begin{proof}
	We build an auxiliary digraph $\tilde{D}$.
	To this end, we start with the union of all the paths in $\Kk\cup \Ll$.
	First, we split every vertex $v$ into two vertices $v^{\variablestyle{in}}$ and $v^{\variablestyle{out}}$ and we add the edge $\Brace{v^{\variablestyle{in}},v^{\variablestyle{out}}}$. Then we replace every edge $e$ of $D$ by the edge $\Brace{\Tail{e}^{\variablestyle{out}},\Head{e}^{\variablestyle{in}}}$. Finally, we add a new vertex $x$ and edges $(x,u^{\variablestyle{in}})$ for each $u\in \Start{\Kk}$ and $(v^{\variablestyle{out}},x)$ for each $v\in \End{\Ll}$.

	Every path $P$ in~$D$ naturally gives rise to a path $\tilde{P}$ in $\tilde{D}$, obtained by applying the edge replacement $e\mapsto \Brace{\Tail{e}^{\variablestyle{out}},\Head{e}^{\variablestyle{in}}}$ to every edge of $P$ and adding all the edges $\Brace{v^{\variablestyle{in}},v^{\variablestyle{out}}}$ for vertices $v$ traversed by $P$. Construct a family of paths $\Pp$ from $\Kk$ by including, for each $P\in \Kk$, the path $\tilde{P}$ with the edge $(x,\Start{P}^{\variablestyle{in}})$ prepended. Similarly, $\Qq$ is constructed from $\Ll$ by including, for each $Q\in \Ll$, the path $\tilde{Q}$ with the edge $(\End{Q}^{\variablestyle{out}},x)$ appended. Thus, the paths of $\Pp$ all start in $x$ and are otherwise vertex-disjoint, and the paths of $\Qq$ all end in $x$ and are otherwise vertex-disjoint.

	\input{figHarvesting}

	Therefore, families $\Pp$ and $\Qq$ satisfy the prerequisites of~\cref{lem:onion_harvesting}.
	We infer that $\tilde{D}$ contains an onion-star of order $k$ with centre $x$.
	We consider the family $\mathcal{Z}$ consisting of the $k$ onions with root $x$ that this onion-star contains.

	We show that every onion in $\mathcal{Z}$ contains a subgraph that corresponds to a \tripod in $D$. Consider an onion $Z\in \mathcal{Z}$. Let $y$ be the \tip of $Z$, let $P^Z_1,P^Z_2$ be the two $x$-$y$-paths of $Z$, and let $Q^Z$ be the $y$-$x$-path of $Z$.
	As $y$ has two distinct in-edges, it is equal to $v_y^{\variablestyle{in}}$ for some vertex $v_y\in \V{D}$.
	Therefore, the subpaths of $P^Z_1,P^Z_2$ obtained by removing their first edges correspond to two paths $P^D_1$ and $P^D_2$ in $D$ that start in two different vertices of $\Start{\Kk}\subseteq S$ and end in $v_y$.
	Similarly, the subpath of $Q^Z$ obtained by removing its last edge corresponds to a path $Q^D$ in $D$ starting in $v_y$ and ending in $\End{\Ll}\subseteq T$.
	To see that $P^D_1, P^D_2$ and $Q^D$ yield a \tripod $R(Z)$ in $D$, note that the only vertex that $P^D_1$ (and, analogously, $P^D_2$) share with $Q^D$ is $v_y$; if both $P^D_1$ and $Q^D$ contain a vertex $v_u \neq v_y$, from the construction it follows that $\Brace{v_u^{\variablestyle{in}},v_u^{\variablestyle{out}}} \in E(P^Z_1) \cap E(Q^Z)$ and $Z$ is not an onion.
	Similarly, since the onions in $\mathcal{Z}$ are pairwise edge-disjoint, it follows that the \tripods $R(Z)$ for $Z\in \mathcal{Z}$ are pairwise vertex-disjoint, and thus $\{R(Z)\colon Z\in \mathcal{Z}\}$ is a packing of $k$ \tripods in $D$.
\end{proof}

\input{ramsey}

\subsection{Constructing \tripods from a path crossing a linkage}
\label{sec:web_or_tripod}

In this section we give another proof of \cref{thm:crossing_linkages}.
As an intermediate statement, we show that in order to find a large packing of \tripods in a \migration digraph, it suffices to find a large linkage from the sources to the sinks and a single path intersecting all the paths in the linkage.

\begin{lemma}
	\label{thm:path_crossing_linkage}
	There is a function $\bound{thm:path_crossing_linkage}{g}{}\colon \N \to \N$ with the following property.
	Let $D$ be a \migration digraph, $k$ be a positive integer,
  $\Pp$ be an $S$-$T$-linkage of size at least $\bound{thm:path_crossing_linkage}{g}{k}$, and $Q$ be a path in $D$ that intersects every path of $\Pp$.
	Then $D$ contains a packing of $k$ \tripods.
\end{lemma}

The proof of \cref{thm:path_crossing_linkage} proceeds by a procedure of iterative extraction of \tripods. More precisely, given the situation described in the statement, we prove that one can find a packing of \tripods directly via an application of \cref{lem:web_to_packing}, or one can extract one \tripod $R$ in such a way that a significant part of the linkage $\Pp$ consists of paths disjoint from $R$, and $Q$ has a long subpath disjoint from $R$. The extraction procedure then proceeds on those remaining parts. A single step of the procedure is encapsulated in the following statement.

\begin{lemma}
	\label{lem:tripod_or_web}
	There exists a function $\bound{lem:tripod_or_web}{g}{} \colon \N \to \N$ with the following property.
	Let $\ell \in \N$, $D$ be a \migration digraph, $\Pp$ be an $S$-$T$-linkage of size $ \bound{lem:tripod_or_web}{g}{\ell}$, and $Q$ be a path intersecting every path of $\Pp$.
	Then $D$ contains
	\begin{enumerate}[label=(P.\roman*)]
		\item \label{lem:tripod_or_web:web} a linkage $\Kk$ with $\Start{\Kk} \subseteq S$ and a linkage $\Ll$ with $\End{\Ll} \subseteq T$ such that $\Abs{\Kk} = \Abs{\Ll} \geq \ell$, and $\Kk$ and $\Ll$ pairwise intersect, or
		\item \label{lem:tripod_or_web:tripod} a \tripod $R$, a sublinkage $\Pp' \subseteq \Pp$ of size at least $\ell$, and a subpath $Q' \subseteq Q$ such that
			\begin{itemize}
				\item $Q'$ intersects every path $P \in \Pp'$, and
				\item $Q'$ and the paths of $\Pp'$ are all vertex disjoint from the \tripod $R$.
			\end{itemize}
	\end{enumerate}
\end{lemma}
\begin{proof}
	We may assume without loss of generality that $\ell\geq 2$. We prove the statement for
	\[\bound{lem:tripod_or_web}{g}{\ell}\coloneqq \ell(\ell+p+p\ell),\]
	where $p\coloneqq \ell^2$.

	First, we construct vertex-disjoint subpaths $A_1,B_1,A_2,B_2,\ldots,A_\ell,B_\ell$ of the path $Q$, placed in this order along $Q$ and called further \emph{segments}, with the following property: If for a segment $X$ we define $\Pp(X)$ to be the sublinkage of $\Pp$ consisting of the paths that intersect $X$, then we require that
	\begin{equation}\label{eq:kapibara}\tag{$\star$}
	|\Pp(A_i)|=\ell+p\ell\quad \textrm{and}\quad |\Pp(B_i)|=p,\qquad\textrm{for all }i\in [\ell].
	\end{equation}
	The segments are constructed by a greedy procedure that extracts them one by one along $Q$.

	The first segment $A_1$ is the shortest prefix of $Q_0\coloneqq Q$ that intersects $\ell+p\ell$ different paths from $\Pp$. Let $Q_0'$ be the suffix of $Q_0$ obtained by removing all the vertices traversed by $A_1$. Then $B_1$ is the shortest prefix of $Q_0'$ that intersects $p$ different paths from $\Pp$. Then we define $Q_1$ to be the suffix of $Q_0'$ obtained by removing all the vertices traversed by $B_1$, and continue with constructing $A_2,B_2,A_3,B_3,\ldots$ in this fashion. Observe that eventually we construct $A_\ell$ and $B_\ell$ without running out of the path $Q$, because $Q$ intersects all the paths in $\Pp$ and we have
	\[|\Pp|=\bound{lem:tripod_or_web}{g}{\ell}=\ell(\ell+p+p\ell).\]

	Note that the sublinkages $\Fkt{\Pp}{A_1},\Fkt{\Pp}{B_1},\ldots,\Fkt{\Pp}{A_\ell},\Fkt{\Pp}{B_\ell}$ are not necessarily disjoint, we only have the cardinality constraints given by~\eqref{eq:kapibara}. In fact, we note that the sublinkages $\Fkt{\Pp}{B_1},\ldots,\Fkt{\Pp}{B_\ell}$ must be almost the same, as formalised below.

	\begin{claim}
		\label{claim:small_symm_diff}
		If there are  distinct $i,j\in [\ell]$ with $\Abs{\Fkt{\Pp}{B_i} \setminus \Fkt{\Pp}{B_j}} \geq \ell$, then there is a tripod fulfilling~\cref{lem:tripod_or_web:tripod}.
	\end{claim}
	\begin{claimproof}
		Choose two distinct paths $P_1, P_2 \in \Fkt{\Pp}{B_i} \setminus \Fkt{\Pp}{B_j}$.
		Assume without loss of generality that the intersection between $P_1$ and $B_i$ that is the first along $P_1$ occurs along $B_i$ before the intersection between $P_2$ and $B_i$ that is the first along $P_2$.
		Then the shortest prefix $P'_1$ of $P_1$ ending in a vertex of $B_i$ together with $P_2$ and the shortest subpath of $B_i$ starting in $\End{P'_1}$ and ending in a vertex of $P_2$ yield a \tripod $R$, see~\cref{fig:tripod_case} for an illustration.

\input{figInduct}

		Next, choose $Q' \coloneqq B_j$ and $\Pp' \coloneqq \Fkt{\Pp}{B_j} \setminus \Fkt{\Pp}{B_i}$. Note that since $\Abs{\Fkt{\Pp}{B_i}}=\Abs{\Fkt{\Pp}{B_j}}=p$, we have \[\Abs{\Pp'}=\Abs{\Fkt{\Pp}{B_j} \setminus \Fkt{\Pp}{B_i}}=\Abs{\Fkt{\Pp}{B_i} \setminus \Fkt{\Pp}{B_j}}\geq \ell.\]
		By construction, both $Q'$ and all paths of $\Pp'$ are disjoint from $R$.
		Thus, $R$, $\Pp'$, and $Q'$ fulfil~\cref{lem:tripod_or_web:tripod}.
	\end{claimproof}

	So we can assume that for all distinct $i,j\in \{1,\ldots,\ell\}$, we have $\Abs{\Fkt{\Pp}{B_i} \setminus \Fkt{\Pp}{B_j}} < \ell$. Denote
	\[\Ll\coloneqq \Fkt{\Pp}{B_1} \cap \cdots \cap \Fkt{\Pp}{B_{\ell}}.\]
	We note the following.

	\begin{claim}
	$\Abs{\Ll} \geq \ell$.
	\end{claim}
	\begin{claimproof}
	 Every element of $\Fkt{\Pp}{B_1}$ that does not belong to $\Ll$ must be contained in $\Fkt{\Pp}{B_1}\setminus \Fkt{\Pp}{B_j}$ for some $j\in \{2,\ldots,\ell\}$. Therefore,
	 \[|\Ll|\geq |\Fkt{\Pp}{B_1}|-\sum_{j=2}^\ell |\Fkt{\Pp}{B_1}\setminus \Fkt{\Pp}{B_j}|\geq p-(\ell-1)\ell=\ell.\qedhere\]
	\end{claimproof}

	Note that $\End{\Ll} \subseteq T$, because $\Ll \subseteq \Pp$. The goal is now to find another linkage of size $\ell$ that starts in $S$ and pairwise intersects with $\Ll$, so that we may apply \cref{lem:web_to_packing}.
	
	To this end, first define
	\[\mathcal{M} \coloneqq \Pp \setminus \left(\Pp(B_1)\cup \ldots\cup \Pp(B_\ell)\right).\]
	We construct an auxiliary \migration digraph $\hat{D}$ consisting of the union of all the paths in $\mathcal{M}$ together with the path $Q$.
	We define \[\Fkt{S}{\hat{D}} \coloneqq \Fkt{S}{D}\qquad\textrm{and}\qquad \Fkt{T}{\hat{D}} \coloneqq \{\End{B_1},\ldots,\End{B_\ell}\}\] for the sources and sinks.

	\begin{claim}\label{cl:linkage-hatD}
	 In $\hat{D}$ there is an $\Fkt{S}{\hat{D}}$-$\Fkt{T}{\hat{D}}$-linkage $\Kk$ of size $\ell$ such that every path of $\Kk$ entirely contains a segment from $\{B_1,\ldots,B_\ell\}$.
	\end{claim}
	\begin{claimproof}
	Towards a contradiction, suppose there is no $\Fkt{S}{\hat{D}}$-$\Fkt{T}{\hat{D}}$-linkage of size $\ell$.
	By Menger's Theorem (\cref{thm:menger}), there exists a set of vertices $F$ of size at most $\ell-1$ such that every  $\Fkt{S}{\hat{D}}$-$\Fkt{T}{\hat{D}}$-path in $\hat{D}$ meets~$F$.
	As $|F|<\ell$, there exists $i \in [\ell]$ such that neither $A_i$ nor $B_i$ is intersected by $F$. Further, we have
	\[|\Fkt{\Pp}{A_i} \cap \mathcal{M}|=|\Fkt{\Pp}{A_i}\setminus \left(\Pp(B_1)\cup \ldots\cup \Pp(B_\ell)\right)|\geq (\ell+p\ell)-p\ell=\ell,\]
	hence there is a path $M\in \mathcal{M}$ that intersects $A_i$ and does not contain a vertex from $F$.
	Thus, the union of $M$, $A_i$, and $B_i$ contains a path from $\Start{M}\in \Fkt{S}{\hat{D}}$ to $\End{B_i} \in \Fkt{T}{\hat{D}}$ that does not contain a vertex from $F$, a contradiction.
	Moreover, since the paths of $\mathcal{M}$ are disjoint from the segments $B_1,\ldots,B_\ell$, every path of $\Kk$ must entirely contain a segment from $\{B_1,\ldots,B_\ell\}$.
	\end{claimproof}

	Let $\Kk$ be the linkage provided by \cref{cl:linkage-hatD}. As $\hat{D}$ is a subgraph of $D$, $\Kk$ is also a linkage in $D$, with start-vertices in $\Fkt{S}{D}$.
	Since every $K \in \Kk$ contains a segment $B_i$, and $B_i$ intersects all the paths of $\Ll$, every path of $\Kk$ intersects every path of $\Ll$. So we may now apply \cref{lem:web_to_packing} to obtain~\cref{lem:tripod_or_web:web}.
\end{proof}

Now we can apply \cref{lem:tripod_or_web} iteratively to derive~\cref{thm:path_crossing_linkage}.

\begin{proof}[Proof of~\cref{thm:path_crossing_linkage}]
	Let $\bound{thm:path_crossing_linkage}{g}{x} \coloneqq
	\bound{lem:tripod_or_web}{g}{}^{k}(\bound{lem:onion_harvesting}{g}{k})$, where $f^{k}$ denotes the $k$-fold composition of a function~$f$.
	We apply~\cref{lem:tripod_or_web} to get the outcome~\cref{lem:tripod_or_web:web} of two pairwise intersecting linkages $\Kk$ and $\Ll$ of order at least $\bound{lem:tripod_or_web}{g}{}^{ k-1}(\bound{lem:onion_harvesting}{g}{k})$ (which can be clearly assumed to be at least $\bound{lem:onion_harvesting}{g}{k}$), or the outcome~\cref{lem:tripod_or_web:tripod} of a \tripod $R_1$ and a remaining linkage $\Pp^1$ of order $\bound{lem:tripod_or_web}{g}{}^{ k-1}(\bound{lem:onion_harvesting}{g}{k})$ and a remaining path $Q^1$.
		
	In case we obtain~\cref{lem:tripod_or_web:web}, then $D$ contains a packing of $k$ \tripods by~\cref{lem:web_to_packing}.
	
	In case we obtain~\cref{lem:tripod_or_web:tripod}, we add the \tripod $R_1$ to a collection of \tripods $\mathcal{R}$ and apply~\cref{lem:tripod_or_web} again on $\Pp^1$ and $Q^1$.
	
	The size of $\Pp$ suffices to repeat this procedure $k$ times.
	So either in one of the applications we obtain~\cref{lem:tripod_or_web:web} and thus a packing of \tripods of size $k$, or we apply~\cref{lem:tripod_or_web} $k$ times in total,  which results in $\mathcal{R}$ becoming a packing of \tripods of size $k$.
\end{proof}

This allows us to infer~\cref{thm:crossing_linkages}.

\begin{proof}[Proof of~\cref{thm:crossing_linkages}]
	We prove the statement for $\bound{thm:crossing_linkages}{g}{k}\coloneqq k(2\bound{thm:path_crossing_linkage}{g}{k}-1).$
	Let $Y \coloneqq \End{\Pp} = \End{\Qq}.$
	If there is a path $Q \in \Qq$ that intersects at least $\bound{thm:path_crossing_linkage}{g}{k}$ paths from $\Pp$, then \cref{thm:path_crossing_linkage} provides a packing of $k$ \tripods in $D$.
	Thus we can assume that every path in $\Qq$ intersects less than $\bound{thm:path_crossing_linkage}{g}{k}$ paths from $\Pp$.
	Similarly, we can assume that every path of $\Pp$ intersects less than $\bound{thm:path_crossing_linkage}{g}{k}$ paths from $\Qq$.
	
	For every $y\in Y$ we write $P^y$ for the unique path of $\Pp$ whose end-vertex is $y$ and $Q^y$ for the unique path of $\Qq$ whose end-vertex is $y$. Construct an auxiliary undirected graph $H$ on the vertex set $Y$ as follows.
	Distinct $y,z\in Y$ are adjacent in $H$ if and only if $P^y$ intersects $Q^z$ or $Q^y$ intersects $P^z$.
	Note that by the observation of the previous paragraph, every vertex of $H$ has degree at most $2\bound{thm:path_crossing_linkage}{g}{k}-2$.
	Hence, $H$ contains an independent set $I$ of size at least $\frac{|V(H)|}{2\bound{thm:path_crossing_linkage}{g}{k}-1}=\frac{|Y|}{2\bound{thm:path_crossing_linkage}{g}{k}-1}=k$.
	
	Observe that for every $y\in I$, the subgraph $P^y \cup Q^y$ contains a \tripod.
	Additionally all those \tripods are pairwise disjoint, because $I$ is independent in $H$. This yields a packing of $k$ \tripods.
\end{proof}

%% file: figOnions.tex
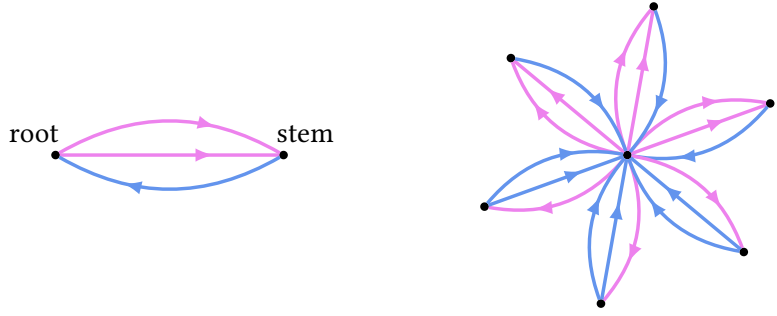
\begin{figure}[!ht]
	\centering
	\begin{tikzpicture}
		\node (onion) at (-3,0) {
			\begin{tikzpicture}
				\node (centre) at (0,0) {};
				\node[vertex] (root) at ($(centre)+(180:1.5)$) {};
				\node[vertex] (stem) at ($(centre)+(0:1.5)$) {};
				\node (stem-label) at ($(stem)+(45:0.4)$) {$\text{stem}$};
				\node (root-label) at ($(root)+(135:0.4)$) {$\text{root}$};

				\draw (root) edge[directededgedecorated,LavenderMagenta] (stem);
				\draw (root) edge[directededgedecorated,LavenderMagenta,bend left] (stem);
				\draw (stem) edge[directededgedecorated,CornflowerBlue,bend left] (root);
			\end{tikzpicture}};
		\node (onion-star) at (3,0) {\begin{tikzpicture}
				\node[vertex] (centre) at (0,0) {};

				\foreach \i in {0,...,5}
				{
					\node[vertex] (s-\i) at ($(centre)+({360/6*\i+20}:2)$) {};
				}
				\foreach \i in {0,1,2}
				{
					\draw (centre) edge[directededgedecorated,LavenderMagenta] (s-\i);
					\draw (centre) edge[directededgedecorated,LavenderMagenta,bend left] (s-\i);
					\draw (s-\i) edge[directededgedecorated,CornflowerBlue,bend left] (centre);
				}
				\foreach \i in {3,4,5}
				{
					\draw (centre) edge[directededgedecorated,LavenderMagenta,bend left] (s-\i);
					\draw (s-\i) edge[directededgedecorated,CornflowerBlue,bend left] (centre);
					\draw (s-\i) edge[directededgedecorated,CornflowerBlue] (centre);
				}
			\end{tikzpicture}};
	\end{tikzpicture}
	\caption{On the left an \emph{onion} and to the right an \emph{onion-star} of order $3$.}
	\label{fig:onion}
\end{figure}

%% file: figHarvesting.tex
	\begin{figure}[!ht]
		\centering
		\begin{tikzpicture}
			\node (centre) at (0,0) {};
			\def\dist{0.5}
			\node[vertex,CornflowerBlue] (s-3) at ($(centre)+(180:2.5)$) {};
			\node[vertex,CornflowerBlue] (s-2) at ($(s-3)+(90:\dist)$) {};
			\node[vertex,CornflowerBlue] (s-1) at ($(s-2)+(90:\dist)$) {};
			\node[vertex,CornflowerBlue] (s-4) at ($(s-3)+(270:\dist)$) {};
			\node[vertex,CornflowerBlue] (s-5) at ($(s-4)+(270:\dist)$) {};

			\foreach \i in {1,...,5}
			{
				\node (k-\i) at ($(s-\i)+(20:3.3)$) {};
				\draw (s-\i) edge[path,CornflowerBlue] (k-\i);

				\node (lp-\i) at ($(k-\i)+(275:2.4)$) {};
			}
			\node (l-1) at ($(k-1)+(170:0.8)$) {};
			\node (l-2) at ($(l-1)+(200:0.43)$) {};
			\node (l-3) at ($(l-2)+(200:0.43)$) {};
			\node (l-4) at ($(l-3)+(200:0.43)$) {};
			\node (l-5) at ($(l-4)+(200:0.43)$) {};

			\node[vertex,LavenderMagenta] (t-3) at ($(centre)+(0:2.5)$) {};
			\node[vertex,LavenderMagenta] (t-2) at ($(t-3)+(90:\dist)$) {};
			\node[vertex,LavenderMagenta] (t-1) at ($(t-2)+(90:\dist)$) {};
			\node[vertex,LavenderMagenta] (t-4) at ($(t-3)+(270:\dist)$) {};
			\node[vertex,LavenderMagenta] (t-5) at ($(t-4)+(270:\dist)$) {};

			\foreach\i in {1,...,5}
			{
				\draw[directededge,dashed,LavenderMagenta] ($(l-\i)$) to[quick curve through ={($(l-\i)+(270:{2*\dist})$) ($(l-\i)+(270:{5*\dist})$) ($(lp-\i)$) ($(t-\i)+(190:1)$)}] ($(t-\i)$);
			}

			\begin{pgfonlayer}{background}
			\def\setdist{0.4}
			\draw[line width=1.5pt,CornflowerBlue] ($(s-1)+(180:\setdist)$) to[closed,curve through ={($(s-1)+(90:\setdist)$) ($(s-1)+(0:\setdist)$) ($(s-1)!0.5!(s-5)+(0:\setdist)$) ($(s-5)+(0:\setdist)$) ($(s-5)+(270:\setdist)$) ($(s-5)+(180:\setdist)$)}] ($(s-5)!0.5!(s-1)+(180:\setdist)$);

			\draw[line width=1.5pt,LavenderMagenta] ($(t-1)+(180:\setdist)$) to[closed,curve through ={($(t-1)+(90:\setdist)$) ($(t-1)+(0:\setdist)$) ($(t-1)!0.5!(t-5)+(0:\setdist)$) ($(t-5)+(0:\setdist)$) ($(t-5)+(270:\setdist)$) ($(t-5)+(180:\setdist)$)}] ($(t-5)!0.5!(t-1)+(180:\setdist)$);
			\end{pgfonlayer}

			\node[vertex,PastelOrange] (x) at ($(centre)+(270:4)$) {};
			\node[PastelOrange!85!black] (x-label) at ($(x)+(300:0.3)$) {$x$};

			\node (sp-5) at ($(s-5)+(225:0.8)$) {};
			\node (sp-4) at ($(sp-5)+(160:0.43)$) {};
			\node (sp-3) at ($(sp-4)+(160:0.43)$) {};
			\node (sp-2) at ($(sp-3)+(160:0.43)$) {};
			\node (sp-1) at ($(sp-2)+(160:0.43)$) {};

			\node (tp-5) at ($(t-5)+(315:0.8)$) {};
			\node (tp-4) at ($(tp-5)+(20:0.43)$) {};
			\node (tp-3) at ($(tp-4)+(20:0.43)$) {};
			\node (tp-2) at ($(tp-3)+(20:0.43)$) {};
			\node (tp-1) at ($(tp-2)+(20:0.43)$) {};

			\begin{pgfonlayer}{background}
			\foreach\i in {1,...,5}
			{
				\draw[directededge,PastelOrange] ($(x)$) to[quick curve through ={($(x)+({190-(\i*5)}:1)$) ($(sp-\i)$)}] ($(s-\i)$);

				\draw[directededgedecorated,PastelOrange] ($(t-\i)$) to[quick curve through ={($(tp-\i)$) ($(x)+({350+(\i*5)}:1)$)}] ($(x)$);
			}
			\end{pgfonlayer}

			\foreach\i in {3,5}
			{
				\draw[line width=1pt,myGrey] ($(x)$) to[quick curve through ={($(x)+({190-(\i*5)}:1)$) ($(sp-\i)$)}] ($(s-\i)$);
			}
			\draw[line width=1pt,myGrey] ($(t-3)$) to[quick curve through ={($(tp-3)$) ($(x)+({350+(3*5)}:1)$)}] ($(x)$);
			\coordinate (A) at ($(s-3)$);
			\coordinate (B) at ($(k-3)$);
			\path [name path=A--B] (A) -- (B);
			\coordinate (C) at ($(l-3)$);
			\coordinate (D) at ($(l-3)+(270:{5*\dist})$);
			\path [name path=C--D] (C) -- (D);
			\coordinate (X) at ($(s-5)$);
			\coordinate (Y) at ($(k-5)$);
			\path [name path=X--Y] (X) -- (Y);
			\path [name intersections={of=A--B and C--D,by=E}];
			\path [name intersections={of=X--Y and C--D,by=F}];

			\draw (s-3) edge[decorate,decoration={snake, segment length=15mm, amplitude=0.7mm},marked] (E);
			\draw (s-5) edge[decorate,decoration={snake, segment length=15mm, amplitude=0.7mm},marked] (F);
			\draw (s-3) edge[path,line width=1pt,myGrey] (E);
			\draw (s-5) edge[path,line width=1pt,myGrey] (F);
			\draw (E) edge[directededge,dashed,line width=1pt,myGrey,out=260,in=100] (F);
			\draw (E) edge[marked,out=260,in=100] (F);

			\draw[directededge,dashed,line width=1pt,myGrey] ($(F)$) to[quick curve through ={ ($(lp-3)$) ($(t-3)+(190:1)$)}] ($(t-3)$);
			\draw[marked] ($(F)$) to[quick curve through ={ ($(lp-3)$) ($(t-3)+(190:1)$)}] ($(t-3)$);

			\node[CornflowerBlue] (K-label) at ($(k-2)+(20:0.3)$) {$\Kk$};
			\node[LavenderMagenta] (L-label) at ($(l-3)+(110:0.3)$) {$\Ll$};
		\end{tikzpicture}

		\caption{
		The two linkages and the new vertex $x$.
		We also illustrate how an \textcolor{myGrey!80!white}{onion} with root $x$ contains a subgraph that corresponds to a \textcolor{myViolet!50!white}{\tripod} in the \migration digraph $D$.}
		\label{fig:web_aux_graph_constr}
	\end{figure}
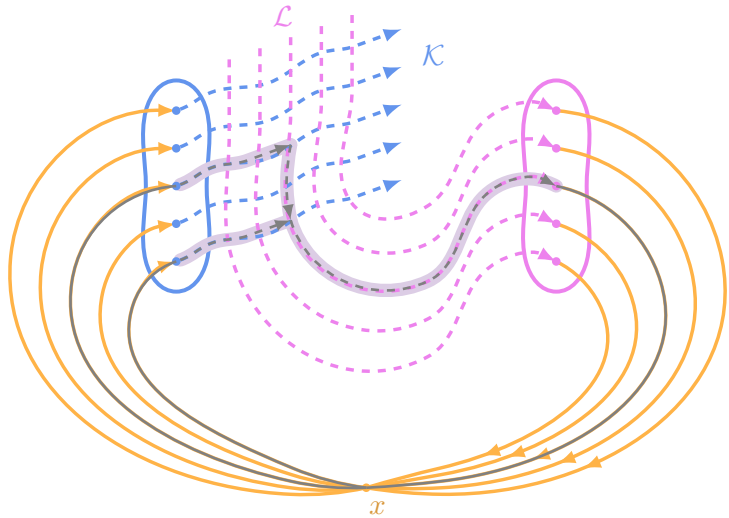

%% file: ramsey.tex
\subsection{Proof using Ramsey arguments}
\label{sec:ramsey}

In this~\namecref{sec:ramsey} we give a direct proof of \cref{thm:crossing_linkages} using Ramsey arguments. We use the following well-known facts. In the second one, a digraph is \emph{semi-complete} if every pair of distinct vertices is connected by at least one edge, and a \emph{transitive tournament} of order $n$ is a digraph with vertex set $\{v_1,\ldots,v_n\}$ and edge set $\{(v_i,v_j)\colon 1\leq i<j\leq n\}$.

\begin{lemma}\label{lem:ramsey}
 There is a function $\Ramsey\colon \N\times \N\to \N$ such that for any two positive integers $a$ and $b$
 every undirected graph on at least $\Ramsey(a,b)$ vertices contains a clique of size $a$ or an independent set of size $b$.
\end{lemma}

\begin{lemma}\label{lem:trans}
 There is a function $\Trans\colon \N\to \N$ such that for any positive integer $c$
 every semi-complete digraph on $\Trans(c)$ vertices contains a transitive tournament of order $c$ as a subgraph.
\end{lemma}

In fact, standard proofs of \cref{lem:ramsey,lem:trans} yield that one can take
\[\Ramsey(a,b)=\binom{a+b-2}{a-1}\qquad\textrm{and}\qquad \Trans(c)=2^{c-1}.\]

\begin{proof}[Proof of \cref{thm:crossing_linkages}]
 We prove the statement for
 \[\bound{thm:crossing_linkages}{g}{k}\coloneqq \Ramsey(\Trans(2\bound{lem:onion_harvesting}{g}{k}),k).\]
 Let $Y\coloneqq \End{\Pp}=\End{\Qq}$; then $|Y|=|\Pp|=|\Qq|\geq \bound{thm:crossing_linkages}{g}{k}$. For $y\in Y$, by $P^y$ and $Q^y$ we denote the unique paths of $\Pp$, respectively $\Qq$, whose end-vertex is $y$.

 Define an auxiliary digraph $\vec H$ on the vertex set $Y$ as follows: for distinct $y,z\in Y$, put an edge $(y,z)$ in $\vec H$ if and only if $P^y$ intersects $Q^z$. By applying \cref{lem:ramsey} with $a\coloneqq \Trans(2\bound{lem:onion_harvesting}{g}{k})$ and $b\coloneqq k$ to the undirected graph underlying $\vec H$, we obtain one of the following outcomes:
 \begin{itemize}
  \item a subgraph $\vec J$ of $\vec H$ on $\Trans(2\bound{lem:onion_harvesting}{g}{k})$ vertices that is semi-complete; or
  \item a subset $Z\subseteq Y$ of size $k$ such that in $\vec H$ there are no edges with both its head and its tail in $Z$.
 \end{itemize}
 In the latter case, we observe that for each $z\in Z$, $P^z\cup Q^z$ is a subgraph of $D$ that contains a \tripod, and all those \tripods are pairwise vertex-disjoint
 yielding the desired packing of $k$ \tripods in $D$.

 We are left with the former case. Applying \cref{lem:trans} to $\vec J$, we find a transitive tournament of order $2\bound{lem:onion_harvesting}{g}{k}$ that is a subgraph of $\vec H$. In other words, there are distinct vertices $y_1,\ldots,y_{2\ell}\in Y$, where $\ell\coloneqq \bound{lem:onion_harvesting}{g}{k}$, such that $(y_i,y_j)$ is an edge in $\vec H$ for all $1\leq i<j\leq 2\ell$.

 Define
 \[\Kk\coloneqq \{P^y\colon y\in \{y_1,\ldots,y_\ell\}\}\qquad\textrm{and}\qquad \Ll\coloneqq \{Q^y\colon y\in \{y_{\ell+1},\ldots,y_{2\ell}\}\},\]
 and observe that
 \begin{itemize}
  \item $|\Kk|=|\Ll|=\ell=\bound{lem:onion_harvesting}{g}{k}$;
  \item $\Start{\Kk}\subseteq S$ and $\End{\Ll}\subseteq T$; and
  \item $\Kk$ and $\Ll$ pairwise intersect, for $(y_i,y_j)$ is an edge in $\vec H$ for all $(i,j)\in \{1,\ldots,\ell\}\times\{\ell+1,\ldots,2\ell\}$.
 \end{itemize}
 Therefore, we may apply \cref{lem:web_to_packing} to conclude that $D$ contains a packing of $k$ \tripods.
\end{proof}

%% file: figInduct.tex
		\begin{figure}[!ht]
			\centering
			\begin{tikzpicture}
				\node (centre) at (0,0) {};
				\node (Q-start) at ($(centre)+(-5,0)$) {};
				\node (Q-end) at ($(centre)+(5,0)$) {};
				\node (Q-start-solid) at ($(Q-start)+(0:1)$) {};
				\node (Q-end-solid) at ($(Q-end)+(180:1)$) {};

				\draw[edge,CornflowerBlue,dashed] (Q-start) to (Q-start-solid);
				\draw[edge,CornflowerBlue,dashed] (Q-end) to (Q-end-solid);
				\draw[edge,CornflowerBlue,name path=Q--1] (Q-start-solid) to (Q-end-solid);

				\node (X-start) at ($(Q-start-solid)+(0.7,0.03)$) {};
				\node (X-end) at ($(X-start)+(2.5,0)$) {};
				\node (Y-start) at ($(X-end)+(1.5,0)$) {};
				\node (Y-end) at ($(Y-start)+(2.5,0)$) {};
				
					\draw[edge,AppleGreen] (X-start.center) to (X-end.center);
					\draw[edge,AppleGreen] (Y-start.center) to (Y-end.center);
					\draw[edge,AppleGreen] ($(X-start)+(90:0.1)$) to ($(X-start)+(270:0.1)$);
					\draw[edge,AppleGreen] ($(X-end)+(90:0.1)$) to ($(X-end)+(270:0.1)$);
					\draw[edge,AppleGreen] ($(Y-start)+(90:0.1)$) to ($(Y-start)+(270:0.1)$);
					\draw[edge,AppleGreen] ($(Y-end)+(90:0.1)$) to ($(Y-end)+(270:0.1)$);

					\node[AppleGreen] (X-label) at ($(X-end)+(240:0.4)$) {$B_i$};
					\node[AppleGreen] (Y-label) at ($(Y-end)+(240:0.4)$) {$B_j$};
					\node[CornflowerBlue] (Q-label) at ($(Q-end)+(60:0.3)$) {$Q$};

					\node (P-1-start) at ($(X-start)+(275:1.7)$) {};
					\node (P-2-start) at ($(P-1-start)+(340:0.6)+(0.5,0)$) {};
					\node (P-2-end) at ($(X-end)+(95:1.7)-(0.5,0)$) {};
					\node (P-1-end) at ($(P-2-end)+(160:0.6)$) {};

					\draw[directededge,dashed,PastelOrange,out=0,in=180,name path=P--1] (P-1-start) to (P-1-end);

					\path [name intersections={of=P--1 and Q--1,by=E}];

					\draw[directededge,dashed,PastelOrange,name path=P--2] (P-2-start) to[quick curve through={($(E)+(1.2,0)$) ($(E)+(1.1,0.7)$) ($(E)+(0.8,0)$) ($(E)+(0.6,-0.7)$) ($(E)+(0.4,0)$) ($(E)+(0.45,0.4)$)}] (P-2-end);

					\path[name path=Q--2] (Q-start-solid) to ($(E)+(0.45,0)$);
					\path [name intersections={of=P--2 and Q--2,by=F}];

					\begin{pgfonlayer}{background}
					\draw[marked,LavenderMagenta!40!white] (E) to (F);
					\draw[marked,LavenderMagenta!40!white,out=0,in=280] (P-1-start) to (E);
					\draw[marked,LavenderMagenta!40!white] (P-2-start) to[quick curve through={($(E)+(1.2,0)$) ($(E)+(1.1,0.7)$) ($(E)+(0.8,0)$) ($(E)+(0.6,-0.7)$) ($(E)+(0.4,0)$) ($(E)+(0.45,0.4)$)}] (P-2-end);
					\end{pgfonlayer}

					\node[PastelOrange] (P-1-label) at ($(P-1-start)+(200:0.3)$) {$P_1$};
					\node[PastelOrange] (P-2-label) at ($(P-2-start)+(200:0.3)$) {$P_2$};
					\node[LavenderMagenta] (R-label) at ($(P-2-end)+(280:0.4)$) {$R$};
			\end{tikzpicture}

			\caption{How two paths $P_1$ and $P_2$ from $\Fkt{\mathcal{P}}{B_i} \setminus \Fkt{\mathcal{P}}{B_j}$ might look and how they yield a \tripod $R$.}
			\label{fig:tripod_case}
		\end{figure}
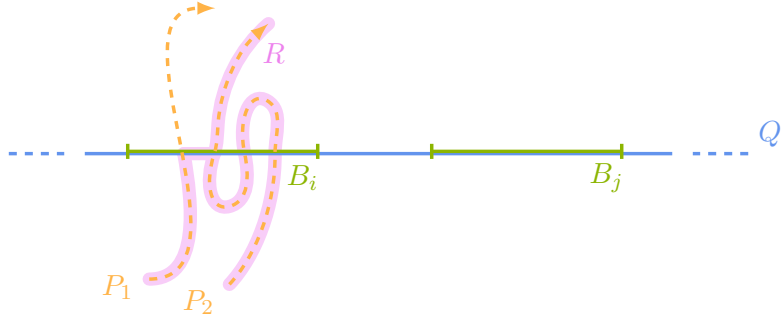

%% file: mainProof.tex
\section{Main proof}\label{sec:mainProof}

With~\cref{thm:crossing_linkages} established, we  proceed with the main proof. First, we need a few definitions.


Let $D$ be a \migration digraph. For a subset of sources $X\subseteq S$ and a subset of sinks $Y\subseteq T$, we shall say that $Y$ is \emph{linkable from} $X$ if there exists an $X$-$Y$-linkage of size $|Y|$, that is, with every vertex of $Y$ being the end-vertex of some path in the linkage.
The next~\namecref{lem:no_large_linkable}, stating that unless we can find a large packing of \tripods, there is no large subset of sinks that is simultaneously linkable from two disjoint subsets of sources, is a direct consequence of \cref{thm:crossing_linkages}.

\begin{corollary}\label{lem:no_large_linkable}
 Let $k$ be a positive integer, $D$ be a \migration digraph with $\packing(D)<k$, $(S_1,S_2)$ be a partition of $S$, and $Y\subseteq T$ be a set linkable from both $S_1$ and from $S_2$.
 Then $|Y|\leq \bound{thm:crossing_linkages}{g}{k}$.
\end{corollary}
%
%

The next~\namecref{lem:cut_partition} is a crucial step in the proof.
We show that if $D$ is a \migration digraph with no large packing of \tripods, then every partition $(S_1,S_2)$ of the sources can be ``projected'' to a partition $(T_1,T_2)$ of the sinks so that all $S_1$-$T_2$-paths and $S_2$-$T_1$-paths can be hit by a small hitting set.

\begin{lemma}\label{lem:cut_partition}
 Let $k$ be a positive integer, $D$ be a migration digraph with $\packing(D) < k$ and $(S_1,S_2)$ be a partition of $S$. Then there exists a partition $(T_1,T_2)$ of $T$ and a vertex subset $F\subseteq V(D)$ with $|F|\leq \bound{thm:crossing_linkages}{g}{k}$ such that $F$ meets all $S_1$-$T_2$-paths and all $S_2$-$T_1$-paths.
\end{lemma}

For the proof of~\cref{lem:cut_partition}, we need a few definitions connected to matroids.
Recall that a \emph{matroid} consists of a \emph{ground set} $U$ (which in this work is always finite) and a family $\cal I$ of subsets of $U$, called \emph{independent sets}, satisfying the following axioms:
\begin{itemize}
 \item $\emptyset\in \cal I$;
 \item if $A\subseteq B$ and $B\in \cal I$, then also $A\in \cal I$; and
 \item if $A,B\in \cal I$ and $|A|<|B|$, then there exists $b\in B\setminus A$ such that also $A\cup \{b\}\in \cal I$.
\end{itemize}
We use a particular type of matroids called \emph{gammoids}, which arise from linkable sets in \migration digraphs.

\begin{theorem}[Perfect~\cite{perfect1968}]\label{thm:gammoid}
 Let $D$ be a \migration digraph. Then the family of subsets of $T$ that are linkable from $S$ forms a matroid with ground set $T$.
\end{theorem}

Recall that if $M=(U,\cal I)$ is a matroid and $X\subseteq U$, then the \emph{rank} of $X$, denoted $\rk_M(X)$, is the largest size of a subset of $X$ that is independent in $M$ (that is, belongs to $\cal I$). The key to the proof of \cref{lem:cut_partition} is the following tight duality theorem, which characterises the maximum cardinality of a set that is simultaneously independent in two matroids with the same ground~set.

\begin{theorem}[Matroid Intersection Theorem~\cite{edmonds1970}]
	\label{thm:matroid_interesection}
 Let $M_1=(U,{\cal I}_1)$ and $M_2=(U,{\cal I}_2)$ be two matroids with the same ground set $U$. Then
 \[\max\,\{\,|I|\colon I\in {\cal I}_1\cap {\cal I}_2\,\}=\min\,\{\,\rk_{M_1}(X)+\rk_{M_2}(Y)\colon (X,Y)\textrm{ is a partition of }U\,\}.\]
\end{theorem}

We may now prove~\cref{lem:cut_partition}.

\begin{proof}[Proof of~\cref{lem:cut_partition}]
 Let $M_1$ and $M_2$ be the matroids with ground set $T$ that are composed of sets linkable from $S_1$, and from $S_2$ respectively; these are indeed matroids by~\cref{thm:gammoid}.
 Let $\ell$ be the largest size of a set that is linkable both from $S_1$ and from $S_2$; or equivalently, that is independent both in $M_1$ and in $M_2$.
 By~\cref{lem:no_large_linkable}, we have $\ell\leq \bound{thm:crossing_linkages}{g}{k}$.

 By~\cref{thm:matroid_interesection}, there is a partition $(T_2,T_1)$ of $T$ such that
 \[\rk_{M_1}(T_2)+\rk_{M_2}(T_1)=\ell.\]
 Note that $\rk_{M_1}(T_2)$ is equal to the maximum size of an $S_1$-$T_2$-linkage.
 By Menger's Theorem (\cref{thm:menger}), there is a vertex subset $F_1$ of size $\rk_{M_1}(T_2)$ that meets all $S_1$-$T_2$-paths.
 Similarly, there is a vertex subset $F_2$ of size $\rk_{M_2}(T_1)$ that meets all $S_2$-$T_1$-paths.
 Now $F\coloneqq F_1\cup F_2$ is a set of size at most $\rk_{M_1}(T_2)+\rk_{M_2}(T_1)=\ell\leq \bound{thm:crossing_linkages}{g}{k}$ that meets all $S_1$-$T_2$-paths and all $S_2$-$T_1$-paths.
\end{proof}

Finally, we have gathered all the tools needed for the proof of our main result.

\begin{proof}[Proof of~\cref{thm:main}]
 We prove the statement for function $\bound{thm:main}{f}{}$ defined inductively as follows:
 \begin{align*}
  \bound{thm:main}{f}{1} & = 0, \\
  \bound{thm:main}{f}{k} & = 2\bound{thm:main}{f}{k-1} + 2\bound{thm:crossing_linkages}{g}{k}\qquad \textrm{for }k\geq 2.
 \end{align*}
 The proof is by induction on $k$. For $k=1$ there is nothing to prove: either there is at least one \tripod and we have the first outcome, or there are no \tripods and we have the second outcome by taking an empty set. So assume $k\geq 2$. We may also assume that $\packing(D)<k$, for otherwise we are done.

 For every partition $(S_1,S_2)$ of $S$, we fix some partition $(T_1,T_2)$ of $T$ and a set $F\subseteq V(G)$ of size at most $\bound{thm:crossing_linkages}{g}{k}$ that is provided for $(S_1,S_2)$ by~\cref{lem:cut_partition}.
 Further, we define \migration digraphs $D_1$ and $D_2$ as follows: for $t\in \{1,2\}$, $D_t$ is obtained from $D$ by deleting the vertices of $F$ and restricting the sources and sinks to $S_t\setminus F$, and $T_t\setminus F$ respectively.
 We say that the partition $(T_1,T_2)$, the vertex subset $F$, and the digraphs $D_1,D_2$ are the objects \emph{associated} with the partition~$(S_1,S_2)$.

 Call a partition $(S_1,S_2)$ of $S$ \emph{splendid} if the following condition holds: if $D_1,D_2$ are the \migration digraphs associated with $(S_1,S_2)$, then
 \[\packing(D_1)>0\qquad\textrm{and}\qquad \packing(D_2)>0.\]
 We argue that if $(S_1,S_2)$ is splendid, then also
 \[\packing(D_1)<k-1\qquad\textrm{and}\qquad \packing(D_2)<k-1.\]
 Let $(T_1,T_2)$ and $F$ be the partition of $T$ and the vertex subset associated with $(S_1,S_2)$.
 Suppose for contradiction that $\packing(D_1)\geq k-1$, hence there is a packing $\cal R$ of $k-1$ \tripods in~$D_1$. As $(S_1,S_2)$ is splendid, we have $\packing(D_2)>0$, so there is also a \tripod in $D_2$, say $R$. Since $F$ meets all $S_1$-$T_2$-paths and all $S_2$-$T_1$-paths, and the vertices of $F$ are deleted both in $D_1$ and in $D_2$, we infer that $R$ is disjoint from all the \tripods in $\cal R$. Hence ${\cal R}\cup \{R\}$ is a packing of $k$ \tripods in~$D$, a contradiction with the assumption $\packing(D)<k$. The argument for the bound $\packing(D_2)<k-1$ is symmetric.

 Hence, if there exists a splendid partition $(S_1,S_2)$ of $S$, say with associated objects $(T_1,T_2)$,~$F$, and $D_1,D_2$, then we may apply induction as follows. Since $\packing(D_1)\leq k-1$, by induction we may find a vertex subset $K_1$ with $|K_1|\leq \bound{thm:main}{f}{k-1}$ that meets all \tripods in $D_1$. Similarly, we find a vertex subset $K_2$ with $|K_2|\leq \bound{thm:main}{f}{k-1}$ that meets all \tripods in $D_2$. It now suffices to~take
 \[K\coloneqq K_1\cup K_2\cup F\]
 and observe that
 \[|K|\leq 2\bound{thm:main}{f}{k-1}+\bound{thm:crossing_linkages}{g}{k}\leq \bound{thm:main}{f}{k}\]
 and that $K$ meets all the \tripods in $D$.

 Therefore, we may assume that there are no splendid partitions of $S$. Hence, for every partition $(S_1,S_2)$ of $S$, say with associated digraphs $D_1,D_2$, we have $\packing(D_1)=0$ or $\packing(D_2)=0$. If the former holds then call $S_1$ an \emph{empty side} of $(S_1,S_2)$, and if the latter holds then call $S_2$ an empty side. Note that it may happen that both sides of a partition are empty, but every partition of $S$ has at least one empty side. We also have the following.

 \begin{claim}\label{cl:empty}
  Let $(S_1,S_2)$ be a partition of $S$ with an empty side $S_1$. Let $F$ be the vertex subset associated with $(S_1,S_2)$. Then $F$ meets all \tripods in $D$ that have at least one source in~$S_1$.
 \end{claim}
 \begin{claimproof}
  Let $D_1,D_2$ be the digraphs associated with $(S_1,S_2)$. We have $\packing(D_1)=0$ by the assumption that the side $S_1$ is empty, so there are no \tripods in $D_1$.

  Let $R$ be a \tripod in $D$, say with sources $s_1,s_2$ and sink $t$. Suppose one of the sources of $R$ belongs to $S_1$, say $s_1$. If $t\in T_2$, then $R$ contains an $S_1$-$T_2$-path and is therefore intersected by $F$. So assume $t\in T_1$. Now, if $s_2\in S_2$, then $R$ contains an $S_2$-$T_1$-path and is therefore intersected by $F$. So assume $s_2\in S_1$ as well. Finally, we conclude that $R$ must be intersected by $F$, because otherwise it would be a \tripod in $D_1$.
 \end{claimproof}

 Now, among all partitions of $S$ pick a partition $(S_1,S_2)$ with a maximum size empty side, say~$S_1$. Let $(T_1,T_2)$, $F$, and $D_1,D_2$ be the objects associated with $(S_1,S_2)$. If $S_1=S$, then by \cref{cl:empty} we have that $F$ meets all the \tripods in $D$, and also $|F|\leq \bound{thm:crossing_linkages}{g}{k}\leq \bound{thm:main}{f}{k}$. Hence, we may assume that there is a vertex $s\in S\setminus S_1$. Consider the partition
 \[(S_1',S_2')\coloneqq (S_1\cup \{s\},S_2\setminus \{s\}).\]
 By the choice of $S_1$ as an empty side with maximum cardinality, we infer that in $(S_1',S_2')$, the side $S_2'$ is empty.

 Since the side $S_1$ of $(S_1,S_2)$ is empty, by~\cref{cl:empty}, we infer that $F$ intersects all the \tripods in $D$ with at least one source in $S_1$. Analogously, the vertex subset $F'$ associated with $(S_1',S_2')$ intersects all the \tripods in $D$ with at least one source in $S_2'$. However, we have
 \[S_1\cup S_2'=S\setminus \{s\},\]
 hence every \tripod in $D$ has a source in $S_1$ or a source in $S_2'$. We conclude that
 \[K\coloneqq F\cup F'\]
 is a vertex subset that meets all \tripods in $D$, and we have
 \[|K|\leq 2\bound{thm:crossing_linkages}{g}{k}\leq \bound{thm:main}{f}{k},\]
 as required.
\end{proof}

%% file: conjecture.tex
\section{Possible extensions}

As discussed in~\cref{sec:intro}, there are very few \EP-type results
concerning acyclic substructures. In this short section we describe a plausible class of acyclic
patterns for which \EP property may hold.

Let \(D\) be a \migration digraph, and let \(k\) be a positive integer.
A \emph{\(k\)-arborescence} in \(D\) is a subgraph of \(D\) which
is an in-arborescence rooted at a vertex \(r \in T(D)\) containing
\(k\) distinct elements of \(S(D)\).
Note that while not every \(2\)-arborescence is a tripod,
every \(2\)-arborescence contains a tripod, and every tripod is a (minimal) \(2\)-arborescence.
Similarly, every \(1\)-arborescence contains an \(S\)-\(T\)-path,
and every \(S\)-\(T\)-path is a (minimal) \(1\)-arborescence.

In these terms, \cref{thm:main} is equivalent to the statement
that \(2\)-arborescences have the \EP property, and Menger's Theorem implies
that \(1\)-arborescences have the \EP property.
This suggests the following conjecture.

\begin{conjecture}\label{conj:arb}
  There exists a function \(\cbound{f}{} \colon \N \times \N \to \N\) such that for every
  \migration digraph \(D\) and positive integers \(k\) and \(\ell\) either
  \begin{itemize}
    \item there are \(\ell\) pairwise vertex-disjoint \(k\)-arborescences in \(D\), or
    \item there exists a subset of at most \(\cbound{f}{k, \ell}\) vertices
      that intersects every \(k\)-arborescence in \(D\).
  \end{itemize}
\end{conjecture}

If one follows the proof of \cref{thm:main} with the goal of extending the result to
the context of arborescences with more sources, we believe that the part of the proof presented in \cref{sec:web_or_tripod}
does indeed work in this more general setting. This does require some work,
in particular one needs to generalise the Onion Harvesting Lemma of~\cite{bozyk2022digraphs} to this setting.
However, one encounters a significant roadblock in trying to apply the Matroid Intersection Theorem.
To be explicit, what we would ideally need to straight-forwardly apply the same ideas would be a statement along the following
lines: there is a function \(\cbound{g}{} \colon \N \times \N \to \N\) such that given positive
integers \(s\) and \(k\), together with \(k\) matroids \(M_1, \dots, M_k\) on a shared ground
set \(U\) one can either find a subset \(I \subseteq U\) of size at least \(s\)
independent in all of \(M_1, \dots, M_k\), or
one can partition the ground set into \(k\) sets \(X_1, \dots X_k\) so that
\(\rk_{M_i}(U \setminus X_i) \leq \cbound{g}{}(s, k)\) for each $i\in [k]$. Unfortunately, this statement is false
already for \(k = 3\); here is an example.

Let \(U = \{0,\dots, 3n - 1\}\) and for $i\in [3]$, let \(M_i\) be the matroid of on ground set $U$ in which the independent sets are all the subsets of 
\(U \setminus \{n(i-1) ,\ldots, n \cdot i-1\}\).
Observe that every $u\in U$ is a \emph{loop} (an element such that $\{u\}$ is not independent) in one of the matroids \(M_i\), hence there is no non-empty common independent set. But if we take
any partition \((X_1, X_2, X_3)\) of $U$ such that \(\rk_{M_1}(U \setminus X_1) \leq C\),
then \(X_1\) must contain at least \(n - C\) elements
from each of the two intervals \([n, 2n)\), and \([2n, 3n)\). Thus \(\rk_{M_i}(U \setminus X_i) \geq \rk_{M_i}(X_1) \geq n - C\)
whenever \(i \neq 1\). Setting $C=n/2$, we see that $\min_{i\in [3]} \rk_{M_i}(U \setminus X_i)\geq n/2$, and $n$ can be chosen arbitrarily.

Therefore, we believe that any proof of~\cref{conj:arb} must clear this hurdle with new ideas, or alternatively use a completely new approach.